\documentclass[12pt]{article}
\usepackage[utf8]{inputenc} % Umlaute
\usepackage{graphicx} %  LaTeX graphics tool when including figure files
\usepackage{amsmath,amssymb}
\usepackage{color,url}
\newtheorem{theorem}{Theorem}% [section]
\newtheorem{definition}{Definition}
\newtheorem{corollary}{Corollary}
\newtheorem{lemma}{Lemma}

\newtheorem{note}{Note}

\def\re{{\mathbb{R}}}

\def\pmatrix#1{\left[\begin{matrix}#1\end{matrix}\right]}
\def\bee{\begin{equation}}
\def\ene{\end{equation}}
\def\beea{\begin{eqnarray}}
\def\enea{\end{eqnarray}}
\def\beeas{\begin{eqnarray*}}
\def\eneas{\end{eqnarray*}}
\def\beas{\begin{eqnarray*}}
\def\eeas{\end{eqnarray*}}

\def\mr#1{(\ref{#1})}

\def\ignore#1{}

\newenvironment{proof}{\begin{list}{$\!\!${\bf Proof.}%
  \rule{1pt}{0pt}}{\setlength{\leftmargin}{0pt}%
  \setlength{\itemindent}{30pt}%
  \setlength{\listparindent}{15pt}}\item}{\rule{0.3em}{0mm}%
  \hfill\framebox[1.2ex]{\rule{0.3em}{0mm}}\end{list}}

\title {Cubic spline functions revisited}

\author{
  \small Florian Jarre, \\
  \small Faculty of Natural Sciences,
  Heinrich Heine Universit{\"a}t D\"usseldorf, Germany
}
\date  {July 7,  2025}
\begin{document}
\maketitle 

\begin{abstract}
  In this paper a fourth order asymptotically optimal 
  error bound for a new cubic interpolating spline function,
  denoted by Q-spline, is derived
  for the case that only function values at
  given points are used but not any derivative information.
  The bound seems to be stronger
  than earlier error bounds for cubic spline interpolation in such setting
  such as the not-a-knot spline.
  A brief analysis of the conditioning of the end conditions of
  cubic spline interpolation leads to a modification of the 
  not-a-knot spline, and
  some numerical examples suggest that the interpolation error
  of this revised not-a-knot spline generally is
  comparable to the near optimal Q-spline and lower than for
  the not-a-knot spline when the mesh size is small.
\end{abstract}


\section{Introduction}

Given knots $x_0<x_1<\ldots <x_n$ and a function $f:[x_0,x_n]\to \re$,
an interpolating function $s$ is searched for with $s(x_i)=f_i:=f(x_i)$
for $0\le i\le n$.
Throughout, it is assumed that $f$ is four times
continuously differentiable on $[x_0,x_n]$ and that $n\ge 5$.

Possibly the most famous result concerning cubic spline functions
is the following Lemma:

\begin{lemma}\label{lemma1}
  The cubic spline function with either the first derivative
  or the second derivative specified at both end points 
  is the unique function that
  minimizes the integral of the square
  of the second derivative among all interpolating $C^2$-functions
  with the same end conditions.
\end{lemma}

When there are no end conditions then the natural spline
$s$ with $s''(x_0)=s''(x_n)=0$ minimizes the integral of the square
of the second derivative among all interpolating $C^2$-functions,
an observation that dates back to \cite{Holladay}.

Lemma \ref{lemma1} is also true for 
the periodic spline when $f$ is periodic with period $x_n\!-\!x_0$; 
see, for example, \cite{Stoer}, Chapter 2.4.
In the following only the case is considered where $f$
is not necessarily periodic.

\begin{itemize}
  \item
%In \cite{HallMeyer} the spline with end conditions \ \ 
%``$s''(x_i)=f''(x_i)$ for $i=0$ and for $i=n$'' \ \ is
%denoted by ``Type II cubic spline interpolant''.
The case $s''(x_i)=0$ for $i=0$ and $i=n$ often is denoted
by {\bf natural spline} in the literature and when
$s'(x_i)=f'(x_i)$ for $i=0$ and $i=n$ are given, this is
often denoted by {\bf clamped spline}. Following this common notation,
the %Type II
cubic spline interpolant %of \cite{HallMeyer}
with end conditions \ \ 
$s''(x_i)=f''(x_i)$ for $i=0$ and for $i=n$ \ \ 
will be denoted by
{\bf clamped natural spline} in the following.
\end{itemize}

Apart from the smoothness of the cubic spline function
established in Lemma \ref{lemma1} also an appealing
error estimate is known.
Setting $\|f^{(4)}\|_\infty:=\max_{\xi\in[x_0,x_n]}|f^{(4)}(\xi)|$
the following error estimate is given in \cite{HallMeyer}:

\begin{lemma}\label{5384}
  For $x\in[x_0,x_n]$ the clamped spline
  and also the clamped natural spline $s$ satisfy
  $$
  |s(x)-f(x)|\le \tfrac{5}{384}\|f^{(4)}\|_\infty h^4
  $$
  with $h:=\max_{1\le i\le n}(x_i-x_{i-1})$. This result is best possible.
\end{lemma}

To put the above lemma into perspective, in Lemma \ref{3128}
below it is compared to a third order
polynomial approximation that uses only function values
and that also allows for an error estimate in terms of
the fourth derivative $f^{(4)}$.

\begin{lemma}\label{3128}
  Let $f$ be four times continuously differentiable on $[x_0,x_n]$.
  Based on the data $f(x_i)$ for $0\le i\le n$,
  %as well as $f'(x_0)$ and $f'(x_n)$,
  an approximation of $f(x)$ for $x\in[x_1,x_{n-1}]$ is possible
  by an interpolating cubic polynomial $p$ such that 
  $$
  |p(x)-f(x)|\le \tfrac{3}{128}\|f^{(4)}\|_\infty h^4
  $$
  with $h:=\max_{1\le i\le n}(x_i-x_{i-1})$.
  
  For $x\in[x_0,x_1]$ or $x\in[x_{n-1},x_n]$ 
  $$
  |p(x)-f(x)|\le \tfrac{1}{24}\|f^{(4)}\|_\infty h^4
  $$
  Both estimates are  best possible.
\end{lemma}

\begin{proof}
  When $x$ coincides with $x_i$ for some $i$ there is nothing to show.
  Let $x\in(x_1,x_{n-1})$. Choose $i$ such that $x\in(x_i,x_{i+1})$ and
  let $p$ be the cubic polynomial interpolating $f$ at
  $x_{i-1},x_i,x_{i+1},x_{i+2}$.
  The standard error bound
  for polynomial interpolation states that
  \bee\label{poly1}
  f(x)-p(x)=\tfrac{f^{(4)}(\xi)}{24} \omega(x) \quad\hbox{with}
  \quad \omega(x) = \prod_{j=i-1}^{i+2}(x-x_j)
  \ene
  and with $\xi\in (x_{i-1},x_{i+2})$.
  Observe that $\|\omega\|_\infty$ increases when, for example,
  $x_{i-1}$ is reduced while $x_i, x_{i+1}$, and $x_{i+2}$ remain
  unchanged. Thus, when maximizing $\|\omega\|_\infty$ one can
  assume without loss of generality that all mesh points have
  maximum distance $x_{k+1}-x_k = h$ for all $k$.
  Straightforward calculations then show that
  $|\omega(x)|\le \tfrac 9{16}h^4$
  for $x\in [x_i,x_{i+1}]$.
  The first statement of the 
  lemma follows when inserting this into \mr{poly1}.

  Similarly, for $x\in [x_0,x_{1}]$, the term $\omega(x)$ is given by
  $\omega(x)= \prod_{j=0}^{4}(x-x_j)$ and straightforward calculations show that
  $|\omega(x)|\le h^4$.
  Likewise also for $x\in[x_{n-1},x_n]$.

  When interpolating the function $f$ with $f(x)\equiv x^4$, the fourth derivative is constant
  so that \mr{poly1} implies that the bounds given in Lemma \ref{lemma1} are best possible.
\end{proof}

For points $x$ near the end points it is not surprising that
the spline approximation using derivative information
at the end points has a lower error estimate than cubic interpolation,
but also for points near the middle of $[x_0,x_n]$
the constant term for the spline approximation
$\tfrac{5}{384}\approx 0.0130 <  0.0234 \approx \tfrac{3}{128}$
is better than for the cubic interpolation. (The constant term
at the end points is $\tfrac{1}{24}\approx 0.0417$.)

Apart from that, in general,
the piecewise cubic interpolation referred to in Lemma \ref{3128}
also is not differentiable at the knots $x_i$ for $1\le i\le n-1$.

\bigskip

Both, Lemma \ref{lemma1} and Lemma \ref{5384} refer to the case
that either $f'$ or $f''$ is known at the end points.
When neither the derivative information for $f$ is available nor
$f$ is known to be periodic, the spline $s$ of choice often is either
the not-a-knot spline or the natural spline.
The natural spline always has second derivative zero at the end points,
$s''(x_0)=s''(x_n)=0$ independent of the second derivative of $f$
at these points. As detailed in Corollary \ref{corollary1} below,
when $f''(x_0)$ or $f''(x_n)$
are nonzero this results in a lower 
approximation accuracy of $f$ by $s$ near $x_0$ or near $x_n$.
Also for the not-a-knot spline there seem to be no error estimates
comparable to Lemma \ref{5384} when an irregular mesh is used.
%It is the aim of the present paper to close this gap to some extent.

Not-a-knot splines on a regular grid with constant distances $x_i-x_{i-1}$
are considered for example in  \cite{Sun_etal}. By an optimal placement of
two additional not-a-knot-nodes, an explicit error bound as in
Lemma \ref{5384} could be derived with constant $10.85/384$
compared to $5/384$ in Lemma \ref{5384}.
%or to $1/384$ for the natural spline on a regular grid.
Earlier, in \cite{deBoor} it was shown that cubic spline interpolation with
the not-a-knot end condition converges to any $C^2$-interpolant 
on arbitrary irregular meshes when the mesh size goes to zero,
but no explicit error rates are given.
Here, an attempt is made to define an interpolating
cubic spline function along with an explicit error estimate
without using any additional points or any
derivative information at the end points.
%Some numerical experiments indicate that this new interpolating
%cubic spline function is very close to the not-a-knot spline in many
%cases.

Before addressing possible replacements of the conditions
for the natural spline or the not-a-knot spline,
the condition number of possible alternative end conditions
is addressed next.

\section{Ill-conditioning of end conditions}\label{sec.ill}
For illustration in this section $n=50$ equidistant mesh points with
distance 1 are considered first.
Since there are two degrees of freedom,
the various interpolating cubic spline functions  always differ
by multiples of two splines $s_1$ and $s_2$ satisfying
$s_1(x_i)=s_2(x_i)=0$ for $0\le i\le n$ and
$$
s_1'(x_0) = 1,\ \ s_1''(x_0) =  2\sqrt{3}, \qquad \hbox{and} \qquad
s_2'(x_0) = 1,\ \ s_2''(x_0) = -2\sqrt{3}.
$$
(The initial values $\pm 2\sqrt{3}$ are chosen at will, what matters is that
$s_1$ and $s_2$ are linearly independent of each other.)

The existence of such functions $s_1$ and $s_2$ implies the following
observation concerning possible generalizations of Lemma \ref{5384}:

\begin{note}\label{note0}
  Without the specification of some form of end condition there does not
  exist any finite number $T=T(f,h)$ such that an interpolating cubic spline
  function $s$ for some given function $f$ and some given mesh size $h$
  always satisfies $\|f-s\|_\infty\le  T$.
\end{note}

(This is so because 
arbitrary multiples of $s_1$ and $s_2$ can be added to an
interpolating cubic spline without changing the
interpolation property.)

On the interval $[x_0,x_1]=[0,1]$ the function $s_1$ takes the form
$$
s_1(x) = (x-x_0) + \sqrt{3}(x-x_0)^2-(1+\sqrt{3})(x-x_0)^3
$$
where the coefficient of $(x-x_0)^3$ is the negative of the sum of the other two
coefficients, so that $s_1(x_0) = s_1(x_1) = 0$. 
From this it follows that
$$
s_1'(x_1) = 1+2\sqrt{3}-3(1+\sqrt{3})=-(2+\sqrt 3)
$$
and
$$
s_1''(x_1) =  2\sqrt{3}-6(1+\sqrt{3})=-6-4\sqrt 3 = -(2+\sqrt 3)\cdot 2\sqrt 3.
$$
Thus, the first and second derivative of $s_1$ at $x_1$ are $-(2+\sqrt 3)$
times the values at $x_0$,
and again, the coefficient of $(x-x_1)^3$ for $s_1$ on the interval $[x_1,x_2]$
is the negative of the sum of the
coefficients for $(x-x_1)$ and  $(x-x_1)^2$.
Inductively, the values of $s_1$ multiply by $-(2+\sqrt 3)$
each time the variable $x$ passes from $[x_{i-1},x_i]$ to $[x_i,x_{i+1}]$.
The graph of $s_1$ oscillates and the absolute values
``explode'' for large values of $x$.

Likewise,
$$
s_2(x) = (x-x_0) - \sqrt{3}(x-x_0)^2+(\sqrt{3}-1)(x-x_0)^3,
$$
with $s_2'(x_1)= %1-2\sqrt{3}+3(\sqrt{3}-1)=
\sqrt 3 - 2$ and $s_2''(x_1)=(\sqrt 3 - 2)s_2''(x_0)$.
Both derivatives are multiplied by $\sqrt 3 - 2\approx -0.268$, and 
the graph of $s_2$ rapidly converges to zero for large values of $x$.

\bigskip

Figure 1  illustrates the exponential growth of $s_1$ and the
exponential decay of $s_2$ for large values of $x$.
The function $s_2$ is scaled by a factor $4\cdot 10^{28}$ to
match the range of $s_1$, and since $s_1(x_i)=s_2(x_i)=0$, not
ln(abs($s_1$)) but ln(abs($s_1$)+1) is plotted; likewise for $s_2$.
(At the knots $x_i$ the values ln(abs($s_1$)+1) go down to zero. In Figure 1
the mesh points used for the plot are chosen disjoint from the knots $x_i$
so that the lines in the figure do not go down to zero at all $x_i$.)

\vfill

\begin{figure}[!h] %Figure 1
%\vskip-10truecm %[width = 20cm]
\vskip-6truecm %[width = 5cm]
\centerline{
  \includegraphics[width = 20cm]{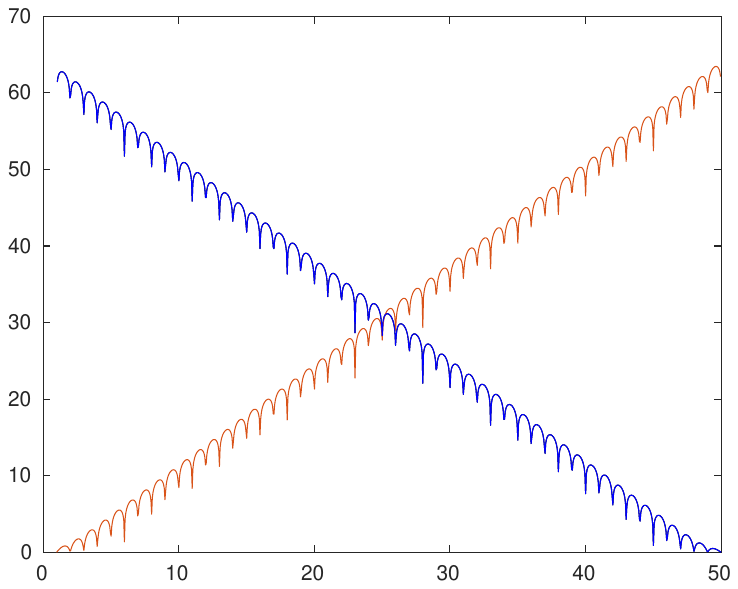}
  }
%  \includegraphics[width = 15cm]{./s1s2}
%\vskip-8truecm %[width = 20cm]
\vskip-8truecm %[width = 20cm]
\caption{Exact values of ln($|s_1|$+1) in red and of ln($4\cdot 10^{28}|s_2|$+1)
  in blue.\\
  The logarithmic scale translates the exponential growth/decay
  of the oscillations\\
  of $|s_1|$ or $|s_2|$ to a linear growth/decay rate.}
\end{figure}

\bigskip

The numerical computation of the coefficients of $s_2$ starting
from $[x_0,x_1]$ and extending to $[x_i,x_{i+1}]$ for $i = 1, 2, \ldots , 50$
is highly unstable. The exact values (derived above) coincide with the data
shown in Figure 1. The numerical values for $s_2$ computed by the above
procedure are depicted in Figure 2, and first behave as predicted but rounding
errors accumulate and the numerical values for $|s_2|$ grow exponentially
for $i\ge 17$. (The errors also grow exponentially for $i\le 17$ but are
still too small to be seen in Figure 2.)
Here, the values of ln(abs($s_2$)+eps) are plotted
where eps is the machine precision so that small values of $s_2$
can be identified on the plot.

\vfill

\begin{figure}[!h] %Figure 2
\vskip-6truecm
\centerline{
  \includegraphics[width = 20cm]{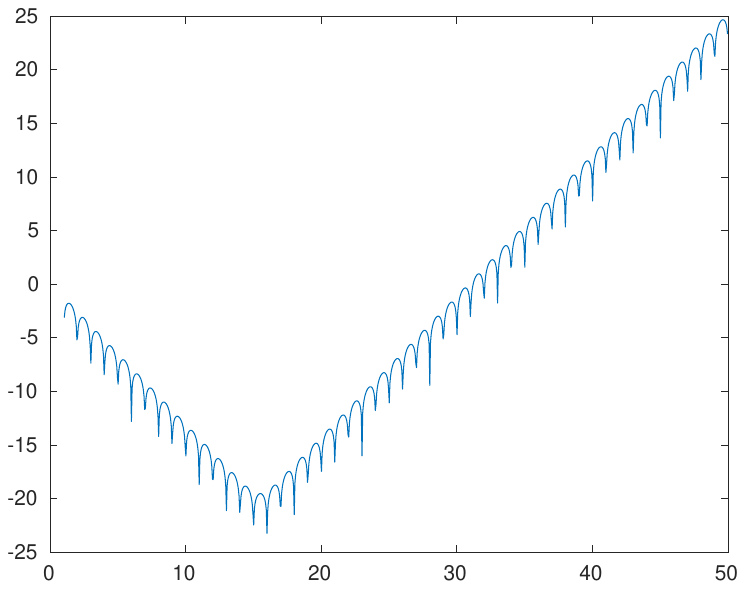}
}
\vskip-8truecm
\caption{Computed values of of ln($|s_2|$+eps) starting
  computations from the left; \\
  near $x=15$ the exponential accumulation of the rounding errors
  becomes visible.}
\end{figure}

To explain this behavior let 
the representation of $s_2$ on the interval $[x_i,x_{i+1}]$ be denoted by
$$
s_2(x)=b_i(x-x_i)+c_i(x-x_i)^2+d_i(x-x_i)^3 \qquad \hbox{for} \ x\in[x_i,x_{i+1}].
$$
Then $d_i=-(a_i+b_i)$ can be treated as an auxiliary variable while
$b_i,c_i$ satisfy the discrete linear dynamical system
$$
\pmatrix{b_{i+1}\\c_{i+1}} = \pmatrix{-2 & -1\\-3 & -2} \pmatrix{b_i\\c_i} =:
\ \ A \pmatrix{b_i\\c_i}.
$$
The eigenvalues of $A$ are just the two numbers $-(2+\sqrt 3)$ and $\sqrt 3 -2$,
and the coefficients of $s_1$ and of $s_2$ yield the associated eigenvectors.
Due to rounding errors, the numerical coefficients converge to
multiples of the eigenvector for the eigenvalue with the larger absolute value;
this is what can be seen in Figure 2.
(The graph of $s_2$ in Figure 1 was computed starting at the
right end point, and the numerical values roughly correspond to the
exact values known from the analysis of the dynamical system.)
% For later use
We note that 
replacing $s(x)$ by $\breve s(x):=s(h^{-1}x)$ for some mesh size $h>0$,
then the $k$-th derivative of $\breve s$ is $\breve s^{(k)}(x)=h^{-k}
s^{(k)}(x)$. The growth factor $-(2+\sqrt 3)$ when moving from
$[x_i,x_{i+1}]$ to $[x_{i+1},x_{i+2}]$ remains the same.

The situation is quite similar when the mesh is not uniform.
In Figure 3, the same number of mesh points was chosen from a
uniform distribution on the same interval. Again, two spline functions $s_1$
and $s_2$ are defined with end values $1$ and $2\sqrt 3$
for the first and second derivative either on
left ($s_1$) or on the right ($s_2$). Again there is some form of
exponential growth
either when $x$ increases, or when $x$ decreases. 

\vfill

\begin{figure}[!h] %Figure 3
\vskip-6truecm
\centerline{
  \includegraphics[width = 20cm]{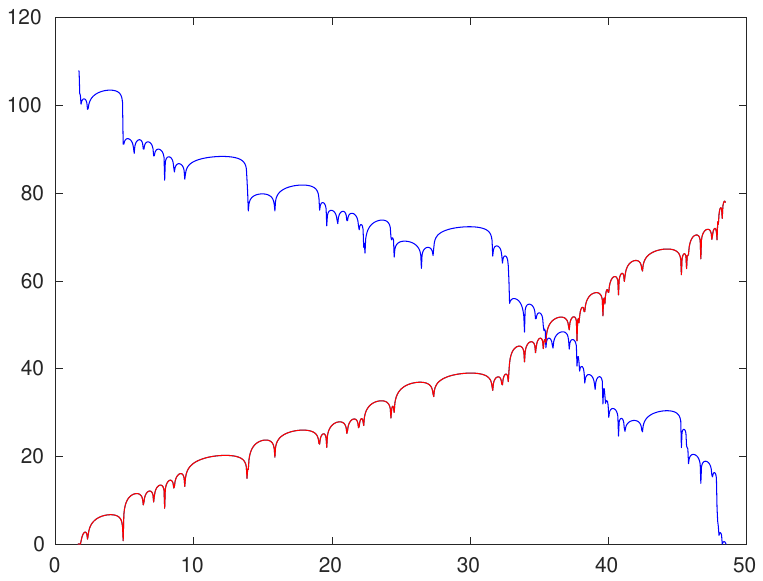}
  }
\vskip-8truecm
\caption{Graph of ln($|s_1|$+1) in red (starting left) and of ln($|s_2|$+1)
  in blue (starting right)\\
  on an irregular mesh.}
\end{figure}

Using slightly different definitions of $s_1$ and $s_2$,
it was observed in \cite{HallMeyer} that 
linear combinations of $s_1$ and $s_2$ generally have large oscillating
function values near $x_0$ and near $x_n$, and comparatively very  small
absolute function values in the middle. Summarizing we obtain the
following observations:

\begin{note}\label{note1}
Finding a spline function $s$ where the values of $s'$ and
$s''$ are given, either both at  $x_0$ or
both at $x_n$, is an extremely  ill-conditioned problem. 
\end{note}

Such ``asymmetric'' end conditions as in Note \ref{note1} will 
not be used in the
sequel; instead ``symmetric'' end conditions will be considered that treat
both ends of the interval $[x_0,x_n]$ the same way.

A second observation can also be made:
\begin{note}\label{note2}
If two interpolating spline functions $\hat s$ and $\bar s$
for a function $f$ on the points $x_0<\ldots<x_n$ are given
with moderate values $\|\hat s-f\|_\infty$ and $\|\bar s-f\|_\infty$ 
then for $x\in[x_0,x_n]$ sufficiently far from
both end points the difference $|\hat s(x) -\bar s(x)|$ is tiny.
\end{note}

Indeed,  $\hat s$ and $\bar s$ differ by a linear combination of
$s_1$ and $s_2$, and by assumption,
the difference has moderate function values near the end points
since else at least one of the values $\|s-f\|_\infty$ would be large. 
As observed above, this linear combination of $s_1$ and $s_2$ has
tiny function values for points $x\in[x_0,x_n]$ sufficiently far from both
end points. Thus, for such $x$ both spline functions $\hat s$ and $\bar s$ 
have similarly strong approximation
properties as stated in Lemma \ref{5384} for the clamped natural spline
-- not only for the function values, but as detailed in \cite{HallMeyer}
also for the first two derivatives.
Without quantifying\footnote{%
An exact quantification can be derived for regular meshes based
on the eigenfunctions $s_1,s_2$ examined in this section.}
this observation exactly, it
will be referred to as {\bf consistent spline property}
in the motivation of the revised not-a-knot spline in Section
\ref{sec.revised}.

\section{Approximating the clamped natural spline}

Lemma \ref{5384} provides an excellent approximation
guarantee for the clamped natural spline when the exact values of
$f''(x_0)$ and $f''(x_n)$ are known. 
This leads to the question in how far approximate values
$\kappa_0$ and $\kappa_n$ used in place of 
$f''(x_0)$ and $f''(x_n)$ lead to splines with
tight approximation guarantees as well. This question is considered next.

\begin{definition}\label{def1}
  Let  $x_0<x_1<\ldots <x_n$ and $f\in C^4([x_0,x_n])$ be given,
  and set $h:=\max_{1\le i\le n}(x_i-x_{i-1})$.
  Further let $\kappa_0,\kappa_n$ be given such that
  $|\kappa_0-f''(x_0)|\le R\|f^{(4)}\|_\infty h^2$ and
  $|\kappa_n-f''(x_n)|\le R\|f^{(4)}\|_\infty h^2$
  for some fixed constant $R$.
  Then the cubic spline $s$ for $f$ on $x_0,\ldots ,x_n$
  with $s''(x_0)=\kappa_0$ and $s''(x_n)=\kappa_n$ is
  called an {\bf $\mathbf R$-approximate clamped natural spline}.  
\end{definition}

\begin{theorem}\label{theorem1}  
  For $x\in[x_0,x_n]$ any $R$-approximate clamped natural spline $s$ satisfies
  $$
  |s(x)-f(x)|\le \left(\tfrac{5}{384}+\tfrac{R}{8}\right)
  \|f^{(4)}\|_\infty h^4.
  $$
\end{theorem}

\begin{proof}
Let $s_{cn}$ be the clamped natural spline and let $s$
be the $R$-approximate clamped natural spline. Setting $s_\Delta:=s_{cn}-s$
it follows from Lemma \ref{lemma1}  for $x\in[x_0,x_n]$ that
$$
|f(x)-s(x)|\le |f(x)-s_{cn}(x)| + |s_\Delta(x)|
\le \tfrac{5}{384}\|f^{(4)}\|_\infty h^4+ |s_\Delta(x)|
$$
To bound $ |s_\Delta(x)|$
let $\mu_i:=(x_{i}-x_{i-1})/(x_{i+1}-x_{i-1})$ and $\lambda_i:=(x_{i+1}-x_{i})/(x_{i+1}-x_{i-1})$
for $1\le i\le n-1$
and denote the second derivatives of $s_\Delta$ at $x_i$ by $M_i:=s''_\Delta(x_i)$
for $0\le i\le n$. In the literature, the quantities
$M_i$ are called moments.
By construction, $s_\Delta(x_i)=0$ for  $0\le i\le n$
and by Definition \ref{def1},
$|M_0|\le R\|f^{(4)}\|_\infty h^2$,
same as for $|M_n|$.
Adapting standard arguments as in Theorem I.3.5 in \cite{Verfuerth},
the (rectangular) linear system for the moments $M_i$
for $s_\Delta$ can be stated as
$$
\left(\begin{matrix}
  \mu_1 & 2 & \lambda_1 & & & \\
  & \mu_2 & 2 & \lambda_2 & & \\
  & & \ddots & \ddots & \ddots & \\ 
  & & & \mu_{n-1} & 2 & \lambda_{n-1}  \\
\end{matrix}\right)
\left(\begin{matrix}
  M_0  \\
  \vdots \\
  \vdots \\
  M_n    
\end{matrix}\right)
=
\left(\begin{matrix}
  0  \\
  \vdots \\
  \vdots \\
  0    
\end{matrix}\right)
$$
where the right hand side follows from   $s_\Delta(x_i)=0$ for all $i$.
Since $M_0$ and $M_n$ are fixed, this is equivalent to 
$$
\left(\begin{matrix}
   2 & \lambda_1 & & & \\
   \mu_2 & 2 & \lambda_2 & & \\
   & \ddots & \ddots & \ddots & \\ 
    &  & \mu_{n-2} & 2 & \lambda_{n-2}  \\ 
    & & & \mu_{n-1} & 2   \\
\end{matrix}\right)
\left(\begin{matrix}
  M_1  \\
  \vdots \\
  \vdots \\
  \vdots \\
  M_{n-1}    
\end{matrix}\right)
=
\left(\begin{matrix}
  -\mu_1 M_0  \\
  0 \\
  \vdots \\
  0 \\
  -\lambda_{n-1}M_n    
\end{matrix}\right)
.
$$
Let the (square) matrix on the left be denoted by
$A$. Then, since $\mu_i,\lambda_i>0$, \ \
$\mu_i+\lambda_i=1$,
the matrix $A$ is strictly diagonally dominant and
$\|Az\|_\infty\ge\|z\|_\infty$ for all $z\in\re^n$.
Hence,  it follows that
$|M_i|\le R\|f^{(4)}\|_\infty h^2$ for all $i$.

Since $s_\Delta$ has zeros at $x_i$ and at $x_{i+1}$ and since its
second derivative is bounded by $R\|f^{(4)}\|_\infty h^2$,
its absolute value on any interval $[x_i,x_{i+1}]$
for $0\le i\le n-1$ cannot exceed
$\tfrac 18 (x_{i+1}-x_i)^2\, R\|f^{(4)}\|_\infty h^2
\le \tfrac{R}{8}\|f^{(4)}\|_\infty h^4$.

\end{proof}

\subsection{Estimating $f''(x_0)$ and $f''(x_n)$}

We begin with a simple estimate of $f''(x_0)$ 
assuming only the continuity of $f^{(4)}$ but not the
existence of higher derivatives:  To this end
the second derivative at $x_0$ of the
cubic interpolant through $x_0,x_1,x_2,x_4$ is  computed.
(Of course, an analogous estimate applies to $f''(x_n)$ as well.)

\begin{lemma}\label{1112}
Let $p$ be the polynomial of degree at most 3 that interpolates
$f$ at $x_0<x_1<x_2<x_3$. If $f$ is four times continuously differentiable
on $[x_0,x_3]$ and $h:=\max_{0\le i\le 2} (x_{i+1}-x_i)$, then
$$
|p''(x_0)-f''(x_0)| = \left|\frac{f^{(4)}(\xi)}{24} 
\left.\omega''(x)\right|_{x=x_0}\right|
\le \left|\frac{11}{12}f^{(4)}(\xi)h^2\right|
$$
where $\xi\in(x_0,x_3)$ and $\omega(x) :=\prod_{j=0}^3(x-x_j)$.
\end{lemma}

\begin{proof}
For completeness a short proof using standard arguments is given:

%\noindent
Straightforward calculations lead to $0<|\omega''(x_0)|\le 22h^2$.
Let $K:=\tfrac{f''(x_0)-p''(x_0)}{\omega''(x_0)}$ and consider the function
$$
\tilde f(x):=f(x)-p(x)-K\omega(x).
$$
By construction, $\tilde f$ has the four zeros $x_0,x_1,x_2,x_3$.
By Rolle's theorem, $\tilde f'$ has three zeros in $(x_0,x_3)$, and $\tilde f''$
has two zeros in  $(x_0,x_3)$. By definition of $K$, also
$\tilde f''(x_0)=0$. Hence $\tilde f'''$ also has two zeros in  $(x_0,x_3)$
and $\tilde f^{(4)}$ also has (at least)
one zero $\xi$ in  $(x_0,x_3)$.
Since $\omega^{(4)}(x)\equiv 24$ it follows that
$$
0=\tilde f^{(4)}(\xi)=f^{(4)}(\xi)-0-24K,
$$
i.e., $K = f^{(4)}(\xi)/24$ or, by definition of $K$,
$$
f''(x_0)-p''(x_0)=\frac{f^{(4)}(\xi)}{24}\omega''(x_0)
$$
from which the claim follows by the bound on $|\omega''(x_0)|$.
\end{proof}

Using the estimates of the above lemma  yields
an $R$-approximate clamped natural spline with $R=\tfrac{22}{24}=\tfrac{11}{12}$.

\subsection{Improving the estimate of $f''(x_0)$ and $f''(x_n)$}
\label{sec.improving}

In view of the proof of Theorem \ref{theorem1},
a sharper approximation of $f''(x_0)$ and $f''(x_n)$ would immediately
result in a sharper error estimate for $\|s-f\|_\infty$.

To start with, only the available interpolation data and the
unknown bound of $\|f^{(4)}\|_\infty$ is used
without assuming the existence of $\|f^{(5)}\|_\infty$.

%The following
%``Münchhausen-heuristic''\footnote{Baron Münchhausen is a legendary
%  literary figure, who claimed he got himself
%  out of a sink hole by pulling at his own hair.}
%aims at doing just that.

For determining the cubic polynomial of Lemma \ref{1112} 
one can compute the Newton interpolation table of divided differences for $f$
with support points $x_0,\ldots,x_3$. Then, a fifth point $x_4$ is added.
The fourth divided difference $f[x_0,x_1,x_2,x_3,x_4]=:\rho$
is the exact value $\tfrac{f^{(4)}(\xi)}{24}$
for some point $\xi\in(x_0,x_4)$.  When forming $\tilde f$ with $\tilde f(x)\equiv f(x)-
\rho(x-x_0)^4$ it follows that $\tilde f(\xi)^{(4)}=0$ and that the first
three derivatives of $ f$ and of $\tilde f$ at $x=x_0$ coincide.
%But the cubic functions that interpolate $f$ and $\tilde f$
%at $x_0,\ldots,x_3$ differ (unless $\rho=0$) and so do their
%derivatives at $x=x_0$.

By construction there is a worst case bound,
$\|\tilde f^{(4)}\|_\infty\le 2\|f^{(4)}\|_\infty$ and since
$\tilde f(\xi)^{(4)}=0$ we may hope that on the interval $[x_0,x_3]$
we have in fact $\|\tilde f^{(4)}\|_\infty < \|f^{(4)}\|_\infty$,
possibly much smaller.
%But in the absence of further information we will never know.

We can then form the cubic interpolation $\tilde c$ of $\tilde f$
on $[x_0,x_1,x_2,x_3]$ and use $\tilde c''(x_0)$ as estimate for
$f''(x_0)$. Likewise for the estimate of $f''(x_n)$.
The approximate  clamped natural spline using these estimates
for $f''(x_0)$ and  $f''(x_n)$ is called   {\bf Q-spline} in the sequel
-- being based on a quartic correction term $-\rho(x-x_0)^4$.
It will at most double  the approximation error $R$
compared to the spline
based on Lemma \ref{1112}, 
but will hopefully reduce it instead.

To quantify this hope one can revisit the divided differences observing that
the first divided difference also satisfies the relation
$$
f[x_0,x_1]:=\tfrac{f[x_1]-f[x_0]}{x_1-x_0}=\int_0^1 f'(x_0+t(x_1-x_0))dt.
$$
It coincides with the value $f'(\xi^{(1)})$ for some $\xi^{(1)}\in [x_0,x_1]$
but it can also be seen as the average value of $f'$ on $ [x_0,x_1]$.
Likewise the $k$-th divided differences form certain average values of
the $k$-th derivatives of $f$ divided by ``$k!$''. 
Estimating the changes of such average value of $f^{(4)}$ can be done "in principle'' without using
the fifth derivative, but this seems to be very tedious.
Since in practical applications the situation is rare that the fourth derivative exists
but the fifth does not, the following simpler analysis assuming the existence of the
fifth derivative is detailed:

Define $\|f^{(5)}\|_\infty:=\infty$ if the fifth derivative
of $f$ is not continuous and else set $\|f^{(5)}\|_\infty$ as maximum absolute value of
$f^{(5)}$ on $[x_0,x_4]$.
Observe that the cubic interpolation $\tilde c$ coincides with
the degree-at-most-4-polynomial $\tilde p$ that
interpolates $\tilde f$ at $x_0,\ldots,x_4$.
(The divided differences for a function $f$ linearly depend on $f$ so that the
fourth divided difference for $\tilde f$ is zero.)
When $\|f^{(5)}\|_\infty$ is finite, 
a proof analogous to the one of Lemma \ref{1112} yields
that 
$$
|\tilde p''(x_0)-f''(x_0)| 
\le R \left|f^{(5)}(\xi)h^2\right|
$$
where $\xi\in(x_0,x_4)$ and $R=R(h) = \tfrac{5h}{12}$.
(The bound on $|\omega''(x_0)|$ is given by $50h^5$ which is divided
by factorial of 5 leading to $\tfrac{5h}{12}h^4$.)

\ignore{
When $f^{(5)}$ is not continuous on $[x_0,x_4]$ a closer consideration of the
divided differences can be used:
The divided differences are initialized as $f^{(0)}_{x_i}(t)\equiv f(x_i)=:f[x_i] $.
The first divided difference is given by
$$
f[x_0,x_1]:=\tfrac{f[x_1]-f[x_0]}{x_1-x_0}=\int_0^1 f'(x_0+t(x_1-x_0))dt.
$$
Setting $f^{(1)}_{x_0,x_1}(t):=f'(x_0+t(x_1-x_0))$, the divided difference $f[x_0,x_1]$
is the average of $f^{(1)}_{x_0,x_1}(t)$ for $t\in[0,1]$.
\hfill\break
$$
f[x_0,x_1,x_2]:=\tfrac{f[x_2,x_1]-f[x_1,x_0]}{x_2-x_0}=
\tfrac 1{x_2-x_0} \left( \int_0^1 f'(x_1+t(x_2-x_1))dt - \int_0^1 f'(x_0+t(x_1-x_0))dt  \right)
$$
$$
=\tfrac 1{x_2-x_0}  \int_0^1 f'(x_1+t(x_2-x_1)) -  f'(x_0+t(x_1-x_0))dt  .
$$
}

Summarizing we obtain the following theorem:

\begin{theorem}\label{theorem2}
  Let $x_0<x_1<\ldots <x_n$ with $n\ge 4$ be given and a four times continuously
  differentiable function $f:[x_0,x_n]\to\re$. 
  Define
  $$
  \|f^{(5)}\|_\infty:=\left\{\begin{matrix}
      \infty && \hbox{if the fifth derivative
        of\ } f \hbox{\ is not continuous} \\
      \max_{x\in [x_0,x_4]\cup[x_{n-4},x_n]} |f^{(5)}(x)| && \hbox{else }
      \end{matrix}\right.
  $$
  (If  $f^{(5)}$ does not exist it is interpreted as not continuous.) 
  Approximate $f''(x_0)$ by $p''(x_0)$ where $p$ is the fourth
  order polynomial interpolating $f$ on $x_0,\ldots,x_4$.
  Likewise for $f''(x_n)$. 
  Let $s$ be the Q-spline using these approximate values
  in place of  $f''(x_0)$ and  $f''(x_n)$.
  For $x\in[x_0,x_n]$ the spline $s$ then satisfies
  $$
  |s(x)-f(x)|\le \left(\tfrac{5}{384}+\tfrac{R}{8}\right)
  \|f^{(4)}\|_\infty h^4.
  $$
  where $R = \min\{\tfrac{11}{6},
  \tfrac{5h\|f^{(5)}\|_\infty}{12\|f^{(4)}\|_\infty}\}$.
\end{theorem}

(In the trivial case that $\|f^{(4)}\|_\infty=0$ it follows that also
$\|f^{(5)}\|_\infty=0$ and the ratio
$\tfrac{5h\|f^{(5)}\|_\infty}{12\|f^{(4)}\|_\infty}$
in Theorem \ref{theorem2} can be replaced with 0.)

The bound in Theorem \ref{theorem2} is not best possible
but it is always a fourth order approximation, and 
when $f^{(5)}$ exists and is continuous at both end points, then for
$h\to 0$ it is arbitrarily close to the best possible bound
derived in \cite{HallMeyer} -- but (!) without using any derivative information.
To our knowledge this is the only explicit fourth order bound
on the error of a cubic spline approximation on an arbitrary 
set of knots in the absence of any
derivative information.

The natural spline in turn only is a second order approximation
as noted in the next corollary.

\begin{corollary}\label{corollary1}
  Under the assumptions of Theorem \ref{theorem2}, the error of the
  natural spline $s$ (with $s''(x_0)=s''(x_n)=0$) is of the exact order
  $h^2$ when $|f''(x_0)|+|f''(x_n)|>0$.
\end{corollary}
\begin{proof}
  Assume without loss of generality that $f''(x_0)=\delta >0$.
  The $O(h^2)$ upper bound for the error can be established
  as in the proof of Theorem \ref{theorem2} using that 
  $|s''(x_0)-f''(x_0)| = |f''(x_0)| = O(1)$ rather than $O(h^2)$.
  For the lower bound observe that for small $h$ the inequality
  $0.8\delta \le f''(x) \le 1.2 \delta$ for $x\in[x_0,x_1]$
  is true. Since $s''$ is linear on $[x_0,x_1]$ with $s''(x_0)=0$,
  either $s''(x)-f''(x) < -0.2 \delta$ for $x\in [x_0,x_0+\frac h4]$
  or $s''(x)-f''(x) > 0.2 \delta$ for $x\in [x_1-\frac h4,x_1]$.
  Both lead to an order $h^2$ maximum value of $|s''(x)-f''(x)|$
  for $x\in[x_0,x_1]$.
\end{proof}

\subsubsection{A heuristic error bound}\label{sec.heur}

When $n\ge 5$, lower bounds for $\|f^{(4)}\|_\infty/24$ and for
$\|f^{(5)}\|_\infty/120$
can be computed by the fourth and fifth divided differences for $f$
near $x_0$.
Using these values as estimates for the true values 
leads to an error bound where the term $R$ in Theorem \ref{theorem2}
can be estimated as
$R\lessapprox\min\left\{\tfrac{11}6,\ \tfrac%
  {25h | f[x_0,x_1,x_2,x_3,x_4,x_5] | }{12|f[x_0,x_1,x_2,x_3,x_4]|}\right\}$
with
$h:=\max_{0\le i\le 4}(x_{i+1}-x_i)$. (Likewise for the other end point $x_n$.)

% xxxxxxxxxx
% the function values at $x_0,\ldots ,x_n$ are already being used
% for the interpolation by the spline function. 
% Is it somehow possible to extract additional information
% from this data so that the spline function can be made unique
% while approximating $f$ as closely as possible?

\section{A revised not-a-knot spline}\label{sec.revised}
By Theorem \ref{theorem2}, for small $h$, the Q-spline is optimal or
near optimal. To be of practical relevance however, a spline must be
found that improves over the not-a-knot spline
({\bf NAK-spline}) in more general situations,
also for large $h$ and irregular data points.
We observe at first that it is impossible to improve over the NAK-spline
in {\em all} situations, because $f$ might just
happen to be equal to the NAK-spline
or might be a very close $C^4$-approximation to the NAK-spline.
%Also when the first sub-intervals $[x_0,x_1]$ or $[x_1,x_2]$ happen to be long
%compared to the remaining sub-intervals, the NAK-spline often is an
%efficient estimate of $f$. 
%(Likewise for last sub-intervals.)
%Nevertheless it is attempted to motivate a spline function that
%leads to better approximations of $f$ in general situations. 
%Since there are many situations to consider -- such as the growth of the
%fourth derivative over the sub-intervals from $x_0$ to $x_4$ and
%the varying sizes of the sub-intervals -- the overall rule considers
%several cases:
The numerical results in the next section show an  
approximation quality of the 
NAK-spline that is comparable to the (nearly optimal)
Q-spline. A possible explanation for this observation might be
the consistent spline property that was observed in Section \ref{sec.ill} \
In the interval $[x_0,x_2]$ the NAK-spline is a cubic interpolating
function with somewhat good approximation properties 
to the first two derivatives of $f$ at $x_2$ due to the consistent
spline property.
To further improve this approximation, the observation can be used that
the fourth divided difference $f[x_0,x_1,x_2,x_3,x_4]=:\rho$
of $f$ generates some average value of $f^{(4)}/24$ on  $(x_0,x_4)$.
If $f^{(4)}$ was constant on $(x_0,x_4)$, the best piecewise
constant approximation $s'''$ of $f'''$ would not require 
$s'''$ to be continuous at $x_1$ (as in the case of the NAK-spline)
but that the jump of $s'''$ at $x_1$ is roughly
given by $\delta_1:=12 \rho  (x_2-x_0)=\tfrac 12 f^{(4)}\cdot (x_2-x_0) $.
(This is the best staircase approximation of a linear function with
slope $f^{(4)}$ for any choice of $x_1\in(x_0,x_2)$.)
The revised not-a-knot-spline ({\bf RNAK-spline})
therefore requires a jump $\delta_1$ of $s'''$ at $x_1$.
And again, similarly for the jump of $s'''$  at $x_{n-1}$.

%The %Q-spline and the
%RNAK-spline will be modified by a safe-guard
%and combined to lead to a revised Q-spline ({\bf RQ-spline})
%leading to a cubic spline function
%that seems to be quite consistently better than
%the NAK-spline:

\medskip

\subsection{A practical safeguard}
For small mesh sizes $x_{i+1}-x_i$ near the end points,
the above modification indeed seems to be an improvement
over the NAK-spline, as illustrated in Section \ref{section.numeric} \
However, for larger mesh sizes, 
it is difficult to state a ``safety-criterion'' for identifying the
cases where the NAK-spline
is in fact better than the RNAK-spline.

If a cubic function is added to $f$ and to the NAK-spline,
the approximation quality does not change. Thus, the safety-criterion cannot
depend on just four function values of $f$,  because any four values can be
interpolated by a cubic function.

A safety-criterion that is independent of
scalings $f(\ .\ )\mapsto \lambda f(\ .\ )$
or $f(\ .\ )\mapsto  f(\lambda \ .\ )$ is based on the heuristic error
estimate in Section \ref{sec.heur}, namely on $5(x_4-x_0)|f_{(5)}|/|f_{(4)}|$
where $f_{(k)}:= f[x_0,\ldots,x_k]$ for $k\ge 1$.
The numerator indicates a possible change of $|f_{(4)}|=|f^{(4)}|/24$
over the interval $[x_0,x_4]$,
and when this quotient is large (or infinity),
the estimated value $f_{(4)}$ is unreliable
and the value $\delta_1$ is shifted towards zero
(since $\delta_1=0$ corresponds to the NAK-spline).

The following somewhat technical safety-criterion
was used in Section \ref{section.numeric}

\noindent
The value of $\rho$ in the definition of the
RNAK-spline is $\rho=f_{(4)}=f^{(4)}(\xi)/24$ with
$\xi\in(x_0,x_4)$. To obtain a more reliable approximation
of the average $f^{(4)}$ on $(x_0,x_2)$, instead of $\xi\in(x_0,x_4)$,
a value $\hat \xi\in(x_0,x_2)$ should be used.
If $f_{(4)}\cdot f_{(5)}>0$, the value $\rho = f_{(4)}$ in the definition
of $\delta_1$ is reduced to $f_{(4)}-5 f_{(5)} (x_2-x_1)$.
(When this results in a sign change of $\rho$
then $\rho$ is set to zero.)

Second, the jump $\delta_1$ is damped further by the factor
$\min\{1,\ \frac{|f_{(4)}|}{5 |f_{(5)}| (x_4-x_0)}\}$ to account for
unreliable estimates of $f^{(4)}$.
Again, likewise for the right end point.

\section{Numerical examples}\label{section.numeric}

In selected numerical examples the not-a-knot spline (NAK-spline)
is compared with the natural spline (NAT-spline, $s''(a)=s''(b)=0$), the 
Q-spline of Theorem \ref{theorem2}, and the RNAK-spline
(revised NAK-spline). The various splines were evaluated by
computing (in parallel) the natural spline $s_{nat}$ and two
zero-interpolating splines $s_1,s_2$ with end conditions 
$s_1''(x_0)=1,\ s_1''(x_n)=0$, and $s_2''(x_0)=0,\ s_2''(x_n)=1$,
and then determining $\alpha,\beta$ to form the final spline
$s_{nat}+\alpha s_1+\beta s_2$.

Tables \ref{table:1} and \ref{table:2} illustrate
the result that the NAT-spline on an interval $[a,b]$
has an excellent approximation guarantee when $f''(a)=f''(b)=0$
but only a second order approximation guarantee
when $f''(a)\not=0$ or $f''(b)\not=0$ (as stated in Corollary \ref{corollary1}).
In comparison, for small $h$, the NAK-spline and the Q-spline display a low error
independent of $f''(a)$ or $f''(b)$. To eliminate effects resulting from irregularities of the mesh,
an equidistant mesh with 6,12,24,48, and 96 knots is considered first.

\vfill

\begin{table}[!h]
\centering
\begin{tabular}{| c| c| c| c| c| c|} 
 \hline
 \# knots & 6 & 12 & 24 & 48 & 96 \\ [0.5ex] 
 \hline\hline
 NAT-spline & 4.5e-4 & 1.8e-5 & 9.1e-7 & 5.2e-8 & 3.1e-09 \\ 
 NAK-spline & 2.7e-3 & 5.5e-5 & 1.4e-6 & 5.2e-8 & 3.1e-09 \\ 
 Q-spline   & 2.2e-3 & 4.0e-5 & 9.6e-7 & 5.6e-8 & 3.1e-09 \\ 
 RNAK-spline& 1.6e-3 & 1.8e-5 & 9.1e-7 & 5.2e-8 & 3.1e-09 \\ 
\hline
\end{tabular}
\caption{Errors $\|s-f \|_\infty$ for $f(x)\equiv \sin(x)$ on $[0,\pi]$, equidistant knots.}
\label{table:1}
\end{table}

(In the last column of Table \ref{table:1} the errors of all four splines
coincided up to 15 digits; the maximum error was in the middle,
where all splines coincide up to machine precision. Near the end points
NAT was best followed by RNAK, Q, and NAK)
%For this example, the fourth derivative of $f$ is zero at both end points, so that
%the correction terms for the Q-spline and the RNAK-spline based on the
%fourth derivative have
%little effect.

%\bigskip

\vfill

\begin{table}[!h]
\centering
\begin{tabular}{| c| c| c| c| c| c|} 
 \hline
 \# knots & 6 & 12 & 24 & 48 & 96 \\ [0.5ex] 
 \hline\hline
 NAT-spline & 1.4e-2  & 2.9e-3 & 6.5e-4 & 1.6e-4 & 3.8e-5  \\
 NAK-spline & 4.3e-3  & 1.7e-4 & 7.9e-6 & 4.3e-7 & 2.5e-8  \\
 Q-spline   & 1.6e-3  & 5.5e-5 & 2.2e-6 & 1.1e-7 & 6.0e-9  \\
 RNAK-spline& 1.6e-3  & 4.6e-5 & 9.1e-7 & 5.2e-8 & 3.1e-9  \\
 \hline
\end{tabular}
\caption{Errors $\|s-f \|_\infty$ for $f(x)\equiv \sin(x)$ on $[\pi/4,5\pi/4]$, equidistant knots.}
\label{table:2}
\end{table}

In Table  \ref{table:3} irregular meshes are considered
with a random uniform distribution scaled such that the endpoints coincide
with the end points of the given interval.
For such irregular meshes, the observation that the natural spline results in a
larger approximation error compared to the other three splines
can be observed in Table \ref{table:3} as well.

\vfill

\begin{table}[!h]
\centering
\begin{tabular}{| c| c| c| c| c| c|} 
 \hline
 \# knots & 6 & 12 & 24 & 48 & 96 \\ [0.5ex] 
 \hline\hline
 NAT-spline & 4.2e-2 & 1.2e-3 & 5.2e-3 & 1.7e-4 & 3.4e-4  \\
 NAK-spline & 4.7e-2 & 1.2e-3 & 5.2e-4 & 1.4e-5 & 8.5e-7  \\
 Q-spline   & 4.6e-2 & 1.2e-3 & 5.2e-4 & 1.4e-5 & 8.5e-7  \\
 RNAK-spline& 4.7e-2 & 1.2e-3 & 5.2e-4 & 1.4e-5 & 8.5e-7  \\
 \hline
\end{tabular}
\caption{Errors $\|s-f \|_\infty$ for $f(x)\equiv \sin(x)$ on $[\pi/4,5\pi/4]$, irregular meshes.}
\label{table:3}
\end{table}

Table \ref{table:3} illustrates the observation from several plots
(not listed here) 
that the maximum error may occur in some sub-interval
$[x_i,x_{i+1}]$ in the middle where $x_{i+1}-x_i$ is large 
and where several or all of the
spline functions almost coincide.

\vfill

\begin{table}[!h]
\centering
\begin{tabular}{| c| c| c| c| c| c|} 
 \hline
 \# knots & 6 & 12 & 24 & 48 & 96 \\ [0.5ex] 
 \hline\hline
 NAT-spline & 1.1e-1 & 1.5e-2 & 1.7e-3 & 1.8e-4 & 5.7e-5  \\
 NAK-spline & 1.4e-1 & 6.2e-3 & 1.3e-3 & 1.8e-4 & 3.6e-5  \\
 Q-spline   & 2.7e-1 & 3.0e-2 & 1.3e-3 & 1.8e-4 & 3.6e-5  \\
 RNAK-spline& 1.4e-1 & 2.1e-2 & 1.3e-3 & 1.8e-4 & 3.6e-5  \\
\hline
\end{tabular}
\caption{Errors $\|s-f \|_\infty$ for $f(x)\equiv 1/(1+x^2)$ on $[-1,3]$, irregular meshes.}
\label{table:4}
\end{table}

The function $f$ considered in Table \ref{table:4} is due to Runge \cite{Runge}
who chose it as an example that polynomial interpolation
may result in high error terms when $f$ has poles
in the complex plane (here at $\pm i$) near the domain of interpolation. 
For $n=12$ the best approximation on the irregular mesh happened to
be given by the NAK-spline illustrating the difficulty in identifying
the cases where the NAK-spline is best and adapting the RNAK-spline
accordingly (without using further knowledge about $f$).
The results of Table \ref{table:4} also
repeat the observation of Table \ref{table:3}
that irregular meshes may produce the maximum error terms
somewhere in the middle where all splines coincide even though
the splines do differ substantially near the end points.
In this respect, random knots
(of course, all four splines were always tested with the same random knots)
are not a good choice for comparing
different spline functions, and
in a final table considering the logistic function, a regular mesh is
used again.

\begin{table}[!h]
\centering
\begin{tabular}{| c| c| c| c| c| c|} 
 \hline
 \# knots & 6 & 12 & 24 & 48 & 96 \\ [0.5ex] 
 \hline\hline
 NAT-spline & 5.5e-3 & 9.6e-4 & 2.1e-4 & 5.1e-5 & 1.2e-5  \\
 NAK-spline & 5.8e-4 & 1.3e-4 & 8.0e-6 & 4.6e-7 & 2.7e-8  \\
 Q-spline   & 2.3e-3 & 1.1e-4 & 8.2e-7 & 1.0e-7 & 6.6e-9  \\
 RNAK-spline& 9.5e-4 & 1.3e-4 & 1.0e-6 & 4.4e-8 & 2.7e-9  \\
\hline
\end{tabular}
\caption{Errors $\|s-f \|_\infty$ for $f(x)\equiv 1/(1+exp(-x))$ on $[-1,4]$, regular meshes.}
\label{table:5}
\end{table}

Table \ref{table:5} with the logistic function $x\mapsto 1/(1+exp(-x))$
also is an example where the large mesh size in the case
$n=5$ (i.e. 6 knots) can lead to a lower approximation error of the
not-a-knot-spline. Identifying such cases where the not-a-knot-spline is best
remains an open question.

For illustration,
Figure 4 displays the difference $f(x)-s(x)$ of the logistic function
and the the
NAK-spline (blue solid line) and the RNAK-spline (red dashed line)
for the case $n=48$.

%\vfill

\begin{figure}[!h] %Figure 4
%\begin{figure}[!h] %Figure 4
%\vskip-10truecm %[width = 20cm]
\vskip-8truecm %[width = 5cm]
\centerline{
\includegraphics[width = 22cm]{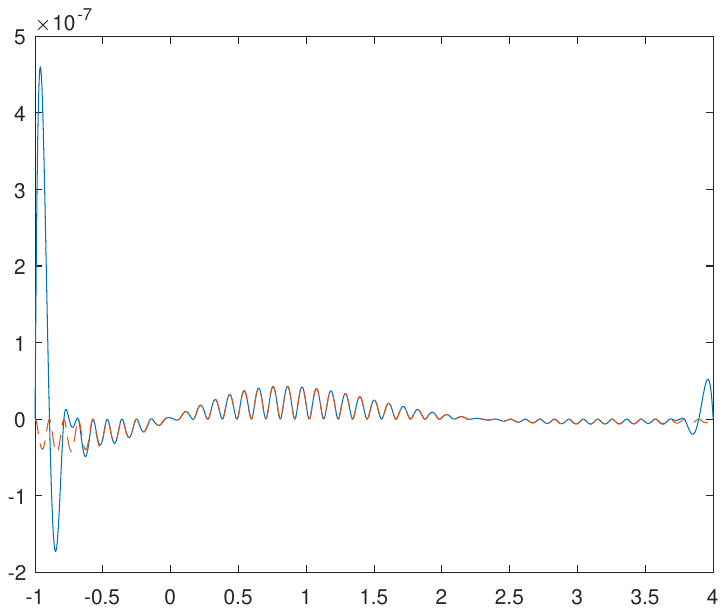}
}
%  \includegraphics[width = 15cm]{./s1s2}
%\vskip-8truecm %[width = 20cm]
\vskip-9truecm %[width = 20cm]
\caption{Graph of $|s_{NAK}-f|$ (blue solid line)
  and of $|s_{RNAK}-f|$
  (red dashed line)\\
  --  for $f(x)\equiv 1/(1+exp(-x))$ on $[-1,4]$ and a regular mesh with $48$ knots.}
\end{figure}

\vfill

Summarizing, the observation in the examples that were tested
is that the NAK-spline
and the Q-spline yield similar errors $\|s-f\|_\infty$ while the
results of the NAT-spline may
be much worse. The NAT-spline is somewhat better than the other splines
if the second derivative at the end points happens to be zero or of small
magnitude. In all examples, 
the approximation quality by the RNAK-spline is never much
worse than by the NAK-spline and for smaller mesh sizes it is
generally a bit better.

%\vfill

\section{Conclusion}

This work arose from an undergraduate class.
It provides a convergence analysis
for cubic spline interpolation at given data points without the use
of derivative information.
%, a topic 
%that is being taught in classes throughout the world.
A near optimal result could be established when
the mesh size is small and the underlying function is five
times continuously differentiable.
Selected numerical examples illustrate the theoretical results
and suggest that the commonly used 
not-a-knot spline
%may have similar
%convergence behavior as the near optimal Q-spline analyzed in this paper.
can be improved for small mesh sizes.


\vfill




\begin{thebibliography}{99}

% \bibitem{Bartels}
% Sören Bartels,
% Numerik 3x9,
% Springer Spektrum Berlin, Heidelberg,  2023.
% DOIhttps://doi.org/10.1007/978-3-662-67497-0
 
% \bibitem{FreundHoppe}
% Roland W. Freund, Ronald H. W. Hoppe,
% Stoer/Bulirsch: Numerische Mathematik 1,
% Springer Berlin, Heidelberg,
% 10. Auflage, 2007.

  
\bibitem{deBoor}
Carl de Boor,    
Convergence of Cubic Spline Interpolation with the Not-A-Knot Condition,
Report Number: MRC-TSR-2876
https://apps.dtic.mil/sti/html/tr/ADA163705/index.html
DTIC Defense Technical Information Center, 1985.


\bibitem{HallMeyer}
Charles A Hall, W.Weston Meyer,
Optimal error bounds for cubic spline interpolation,
Journal of Approximation Theory
Volume 16, Issue 2, February 1976, Pages 105-122.


\bibitem{Holladay}
John C. Holladay,
A Smoothest Curve Approximation,
Mathematical Tables and Other Aids to Computation,
Vol. 11, No. 60, pp. 233-243 (11 pages)
American Mathematical Society, 1957.


%\bibitem{jarregithub}
%Github repository XXXX to be completed

\bibitem{Runge}
Carl Runge,
Über empirische Funktionen und die Interpolation zwischen
äquidistanten Ordinaten,
Zeitschrift für Mathematik und Physik, 1901.
%XXXX noch nachschlagen!!!


\bibitem{Stoer}
Joseph Stoer, Roland Bulirsch,
Introduction to Numerical Analysis, 
Third Edition 2002,
Springer,Texts in Applied Mathematics, Vol. 12, 2002.
%ISBN-10 ‏ : ‎ 038795452X
%ISBN-13 ‏ : ‎ 978-0387954523 
% R. Bartels (Übersetzer), W. Gautschi (Übersetzer), C. Witzgall (Übersetzer)


\bibitem{Sun_etal}
Meng Sun, Lin Lan, Chun-Gang Zhu , Fengchun Lei,
Cubic spline interpolation with optimal end conditions,
Journal of Computational and Applied Mathematics 425,
Article 115039, 2023.


\bibitem{Verfuerth}
  Rüdiger Verfürth,
  Einführung in die Numerische Mathematik,
  Vorlesungsskriptum Sommersemester 2018,
  Fakultät für Mathematik, Ruhr-Universität Bochum,
  https://homepage.ruhr-uni-bochum.de/ruediger.verfuerth/files/lectures/EinfNumerik.pdf,
  Accessed, June 1, 2025.
  

\end{thebibliography}
\end{document}